\documentclass[11pt,reqno,letterpaper]{amsart}

\usepackage{amssymb,amsmath,amsthm,amsfonts}
\usepackage{bbm,enumerate}
\usepackage{xcolor}

\usepackage{bookmark}
\usepackage{hyperref}
\hypersetup{pdfstartview={FitH}}

\theoremstyle{plain}
\newtheorem{theorem}{Theorem}
\newtheorem{lemma}[theorem]{Lemma}
\newtheorem{corollary}[theorem]{Corollary}

\theoremstyle{remark}
\newtheorem{remark}{Remark}

\numberwithin{equation}{section}

\renewcommand{\leq}{\leqslant}
\renewcommand{\geq}{\geqslant}

\begin{document}

\title{On binomial sums, additive energies, and lazy random walks}

\author{Vjekoslav Kova\v{c}}
\address{Department of Mathematics, Faculty of Science, University of Zagreb, Bijeni\v{c}ka cesta 30, 10000 Zagreb, Croatia}
\email{vjekovac@math.hr}



\begin{abstract}
We establish a sharp estimate for $k$-additive energies of subsets of the discrete hypercube conjectured by de Dios Pont, Greenfeld, Ivanisvili, and Madrid, which generalizes a result by Kane and Tao.
This note proves the only missing ingredient, which is an elementary inequality for real numbers, previously verified only for $k\leq100$.
We also give an interpretation of this inequality in terms of a lazy non-symmetric simple random walk on the integer lattice.
\end{abstract}

\maketitle


\section{Introduction}
The main contribution of the present paper is verification of the following elementary inequality conjectured by de Dios Pont, Greenfeld, Ivanisvili, and Madrid \cite{DGIM21}.

\begin{theorem}\label{thm:main}
For every positive integer $k$ and for $a,b\in[0,\infty)$ we have
\begin{equation}\label{eq:mainineq}
\sum_{j=0}^{k} \dbinom{k}{j}^2 a^{p_k(k-j)/k} b^{p_k j/k} \leq (a+b)^{p_k},
\end{equation}
where
\begin{equation}\label{eq:defofpk}
p_k=\log_2\binom{2k}{k}.
\end{equation}
\end{theorem}

By taking $a=b=1$ one clearly sees that formula \eqref{eq:defofpk} gives the smallest possible exponent $p_k$ such that estimate \eqref{eq:mainineq} can hold.
The particular case $k=2$ of this estimate was established by Kane and Tao \cite[Lemma~8]{KT17}, while the authors of \cite{DGIM21} verified it for all $k\leq10$ in \cite[Lemma~5]{DGIM21}, with a comment that they also performed verification for $k\leq100$ with an aid of a computer \cite[Remark~12]{DGIM21}.
The paper \cite{KT17} calls \eqref{eq:mainineq} simply \emph{an elementary inequality}, while \cite{DGIM21} also calls it \emph{a subtle
inequality for the Legendre polynomials}.
Namely, if $P_k$ denote the Legendre polynomials,
\[ P_k(z) = \frac{1}{2^k} \sum_{j=0}^k \binom{k}{j}^2 (z+1)^{k-j} (z-1)^{j}, \]
see \cite[Formula~18.5.8]{NIST}, then inequality \eqref{eq:mainineq} can be rephrased as
\begin{equation*}
P_k(z) \leq \bigg( \Big(\frac{z+1}{2}\Big)^{k/p_k} + \Big(\frac{z-1}{2}\Big)^{k/p_k} \bigg)^{p_k} \quad\text{for } z\in[1,\infty);
\end{equation*}
see the details in \cite{DGIM21}.

\smallskip
The importance of Theorem~\ref{thm:main} comes from the following application in additive combinatorics.
For a positive integer $k$ the notion of \emph{$k$-additive energy} $E_k(A)$ of a finite set $A\subset\mathbb{Z}^d$ was defined in the paper \cite{DGIM21} as the number of $2k$-tuples $(a_1,\ldots,a_{2k})\in A^{2k}$ such that $a_1+\cdots+a_{k}=a_{k+1}+\cdots+a_{2k}$.
In the particular case $k=2$ this specializes to the well-known concept of the \emph{additive energy}; see \cite[Section~2.3]{TV06}.
Let $|A|$ denote the cardinality of $A$.

\begin{corollary}\label{cor:additive}
Take positive integers $d,k$ and let $p_k$ be as in \eqref{eq:defofpk}. For every set $A\subseteq\{0,1\}^d\subset\mathbb{Z}^d$ we have
\begin{equation}\label{eq:mainadditive}
E_k(A)\leq|A|^{p_k}.
\end{equation}
\end{corollary}

By taking $A=\{0,1\}^d$ one again sees that $p_k$ is the smallest possible exponent such that inequality \eqref{eq:mainadditive} can hold.
This sharp estimate was conjectured by de Dios Pont, Greenfeld, Ivanisvili, and Madrid, who also showed how it can be derived from \eqref{eq:mainineq} by using it in the step of the mathematical induction on the dimension $d$; see \cite[Section~3]{DGIM21}. The same deduction was previously performed for $k=2$ by Kane and Tao \cite[Section~2]{KT17}.
Thus, prior to the present paper, Corollary~\ref{cor:additive} has only been confirmed for small values of $k$, namely $k\leq100$; see \cite[Theorem~7]{KT17} and \cite[Theorem~3]{DGIM21}.

\smallskip
We use the opportunity to also give a probabilistic reformulation of inequality \eqref{eq:mainineq}, which will not be needed in its proof, but it might be interesting on its own. Suppose that $X_1,X_2,\ldots$ are independent identically distributed random variables that take values in $\{-1,0,1\}$ and satisfy $\mathbb{P}(X_1=0)=1/2$. Let $S_n=X_1+\cdots+X_n$ denote the associated random walk on $\mathbb{Z}$ starting at $0$. Such a process is often called \emph{a lazy simple random walk}, as it only makes a move with probability $1/2$; see for instance \cite{H01,MP12}.
Note that the distribution of this discrete stochastic process $(S_n)_{n=0}^\infty$ is uniquely determined by a single number, namely the probability of taking a step to the right, $\mathbb{P}(X_1=1)\in[0,1/2]$.

\begin{corollary}\label{cor:probab}
For every lazy simple random walk $(S_n)_{n=0}^\infty$ on $\mathbb{Z}$, every positive integer $k$, and for the number $p_k$ defined by \eqref{eq:defofpk}, the estimate
\begin{equation}\label{eq:probineq}
\mathbb{P}(S_k=0)^{1/p_k} \leq \mathbb{P}(S_k=-k)^{1/p_k} + \mathbb{P}(S_k=k)^{1/p_k}
\end{equation}
holds. Moreover, $p_k$ from \eqref{eq:defofpk} is the smallest number such that \eqref{eq:probineq} holds for a fixed $k$ and every lazy simple random walk $(S_n)_{n=0}^\infty$.
\end{corollary}

The number $\mathbb{P}(S_k=-k)$ (resp.\@ $\mathbb{P}(S_k=k)$) can be interpreted as the probability that the first $k$ steps of the random walk are all made to the left (resp.\@ right). Clearly, $\mathbb{P}(S_k=0)$ is the probability that, after $k$ steps, the random walk ends up at its starting point.

\smallskip
The proof of Theorem~\ref{thm:main} is given in Section~\ref{sec:theorem}.
It only needs basic single-variable calculus. A crucial ingredient is a second-order ordinary differential equation coming from the classical differential equation for the Legendre polynomials, even though we do not work with special polynomials at all.

Mathematica \cite{Mathematica} is used extensively in two different ways.
First, finitely many inequalities for concrete real numbers are verified as parts of the proofs of Lemmata~\ref{lm:easyest} and \ref{lm:cisnegative}, by computing relevant numerical expressions using infinite precision (i.e., tracking the propagation of the numerical error) and sufficient accuracy (which is in all of our cases chosen to be $10$ accurate digits after the leading zeros).
In other words, a numerical expression \verb+e+ is always approximated using the command \verb+N[e,{+$\infty$\verb+,10}]+.
Second, symbolic differentiation via the command \verb+D+ and algebraic simplification via the command \verb+Simplify+ are used in the proof of Lemma~\ref{lm:differential}.
All these operations are perfectly reliable. Note that we do not rely on testing infinitely many inequalities for real numbers, or on any sketches of graphs of functions.

Corollary~\ref{cor:probab} is established in Section~\ref{sec:corollary} by showing that \eqref{eq:probineq} is just another restatement of the elementary inequality \eqref{eq:mainineq}.
Finally, Section~\ref{sec:means} gives yet another reformulation of \eqref{eq:mainineq}, in terms of the means of a pair of nonnegative numbers, suggested to the author by Jairo Bochi.


\section{Proof of Theorem~\ref{thm:main}}
\label{sec:theorem}
For each positive integer $k$ let us define a function $f_k\colon[0,1]\to[0,\infty)$ by the formula
\begin{equation}\label{eq:mdefoffk}
f_k(x) := \sum_{j=0}^{k} \dbinom{k}{j}^2 (1-x)^{p_k(k-j)/k} x^{p_k j/k}.
\end{equation}
When $x=0$ or $x=1$, the expression $0^0$ is interpreted as $1$, as is common in relation with discrete sums.
The desired inequality \eqref{eq:mainineq} is homogeneous of order $p_k$ in $a,b$. Thus, one can additionally assume $a+b=1$ and then \eqref{eq:mainineq} simply reads
\begin{equation}\label{eq:mainineq2}
f_k(x)\leq 1 \quad\text{for every } x\in[0,1].
\end{equation}
Note that \eqref{eq:mainineq2} is trivial for $k=1$, since $f_1$ is identically equal to $1$. Throughout this section we assume that $k\geq2$ is a fixed integer.

The following exact form of Stirling's formula is shown in \cite{Rob55}:
\begin{equation}\label{eq:Stirling}
(2\pi)^{1/2} n^{n+1/2} e^{-n+1/(12n+1)} < n! < (2\pi)^{1/2} n^{n+1/2} e^{-n+1/12n}
\end{equation}
for every positive integer $n$.
From \eqref{eq:Stirling} we get
\[ \frac{2^{2k}}{\sqrt{\pi k}} e^{-1/6k} < \binom{2k}{k} < \frac{2^{2k}}{\sqrt{\pi k}}, \]
so taking logarithms and using $(\log_2 e)/6<1/4$ yields
\begin{equation}\label{eq:ineqforpk}
2k - \frac{1}{2}\log_2(\pi k) - \frac{1}{4k} < p_k < 2k - \frac{1}{2}\log_2(\pi k).
\end{equation}
In particular, we certainly have
\begin{equation}\label{eq:ineqforpk2}
k < p_k < 2k - 1.
\end{equation}

\begin{lemma}\label{lm:easyest}
For every $x\in[0,1/10]$ and $j\in\{0,1,\ldots,k\}$ we have
\begin{equation}\label{eq:binomials}
\dbinom{k}{j} (1-x)^{p_k(k-j)/k} x^{p_k j/k} \leq (1-x)^{k-j} x^j.
\end{equation}
Consequently, for $x\in[0,1/10]$ we also have $f_k(x) \leq 1$.
\end{lemma}

\begin{proof}
The first claim is trivial when $j=0$ or $j=k$ because of \eqref{eq:ineqforpk2}, so assume $1\leq j\leq k-1$.
The desired estimate \eqref{eq:binomials} can be written equivalently as
\begin{equation}\label{eq:smart1}
\theta(x) \geq 0 \quad \text{for } 0<x\leq1/10,
\end{equation}
where $\theta\colon(0,1)\to\mathbb{R}$ is defined by
\[ \theta(x) := -\Big(\frac{p_k}{k}-1\Big) (k-j) \log(1-x) -\Big(\frac{p_k}{k}-1\Big) j \log x - \log\dbinom{k}{j}. \]
From
\[ \theta'(x) = \Big(\frac{p_k}{k}-1\Big) \Big(\frac{k-j}{1-x}-\frac{j}{x}\Big) \]
we see that $\theta$ is decreasing on $(0,j/k]$ and increasing on $[j/k,1)$.

\smallskip
\emph{Case 1: $j\leq k/10$.}
Note that this case is, in fact, void unless $k\geq10$.
In this case, \eqref{eq:smart1} is equivalent with nonnegativity of $\theta$ at the point of its global minimum, namely $\theta(j/k)\geq0$, which transforms back into
\begin{equation}\label{eq:smart2}
\dbinom{k}{j} \leq \Big( \frac{k^k}{j^j (k-j)^{k-j}} \Big)^{p_k/k-1}.
\end{equation}
Stirling's formula \eqref{eq:Stirling} together with an easy inequality
\[ \frac{1}{12k} - \frac{1}{12j+1} - \frac{1}{12(k-j)+1} \leq 0 \]
gives
\begin{equation}\label{eq:Stirbinom}
\dbinom{k}{j} < \frac{1}{\sqrt{2\pi}} \frac{k^{k+1/2}}{j^{j+1/2} (k-j)^{k-j+1/2}}.
\end{equation}
Because of this, we will have \eqref{eq:smart2} once we can show
\[ \frac{1}{\sqrt{2\pi}} \frac{k^{k+1/2}}{j^{j+1/2} (k-j)^{k-j+1/2}} \leq \Big( \frac{k^k}{j^j (k-j)^{k-j}} \Big)^{p_k/k-1}, \]
but this transforms into
\[ \Big(\frac{j}{k}\Big)^{(2k-p_k)j/k+1/2} \Big(1-\frac{j}{k}\Big)^{(2k-p_k)(k-j)/k+1/2} \geq \frac{1}{\sqrt{2\pi k}}. \]
Thanks to the constraint $1\leq j\leq k/10$, by substituting $t=j/k$ verification of the last display reduces to
\[ t^{(2k-p_k)t+1/2} (1-t)^{(2k-p_k)(1-t)+1/2} \geq \frac{1}{\sqrt{2\pi k}} \quad\text{for } t\in\Big[\frac{1}{k},\frac{1}{10}\Big], \]
i.e., to
\begin{equation}\label{eq:smart4}
\phi_k(t) \geq0 \quad\text{for } t\in\Big[\frac{1}{k},\frac{1}{10}\Big],
\end{equation}
where
\[ \phi_k(t) := \Big((2k-p_k)t+\frac{1}{2}\Big) \log t + \Big((2k-p_k)(1-t)+\frac{1}{2}\Big) \log(1-t) + \frac{1}{2}\log(2\pi k). \]
Differentiating
\begin{align*}
\phi_k'(t) & = -(2k-p_k) \log\frac{1-t}{t} + \frac{1-2t}{2(1-t)t}, \\
\phi_k''(t) & = \frac{2k-p_k}{(1-t)t} - \frac{2t^2-2t+1}{2(1-t)^2 t^2}
\end{align*}
and denoting
\[ t_k := \frac{1}{2}-\frac{1}{2}\sqrt{1-\frac{2}{2k+1-p_k}} \]
we easily see, thanks to \eqref{eq:ineqforpk} and $k\geq10$:
\begin{align*}
& \text{$\phi_k$ is concave on $(0,t_k]$}; \\
& \text{$\phi_k'$ is increasing on $[t_k,1/2]$}; \\
& \phi_k'\Big(\frac{1}{10}\Big) < - \frac{\log_2(10\pi)}{2}\log 9 + \frac{40}{9} < 0.
\end{align*}
From these three claims we conclude that the verification of \eqref{eq:smart4} reduces to
\begin{equation}\label{eq:smart5}
\phi_k\Big(\frac{1}{k}\Big) \geq0
\end{equation}
and
\begin{equation}\label{eq:smart6}
\phi_k\Big(\frac{1}{10}\Big) \geq0.
\end{equation}
Namely, if $t_k\geq1/10$, then just from the concavity of $\phi_k$ on $(0,t_k]$ we get
\begin{equation}\label{eq:refclear}
\phi_k(t) \geq \min \Big\{\phi_k\Big(\frac{1}{k}\Big), \phi_k\Big(\frac{1}{10}\Big) \Big\}
\end{equation}
for every $t\in[1/k,1/10]\subset(0,t_k]$.
Next, if $1/k\leq t_k<1/10$, then this concavity only gives
\begin{equation*}
\phi_k(t) \geq \min \Big\{\phi_k\Big(\frac{1}{k}\Big), \phi_k(t_k) \Big\}
\end{equation*}
for every $t\in[1/k,t_k]$.
However, since $\phi_k'$ is increasing on $[t_k,1/10]$ and still negative on the right end of that interval, we conclude that $\phi_k$ is, in fact, decreasing on that same interval. Consequently,
\[ \phi_k(t) \geq \phi_k\Big(\frac{1}{10}\Big) \]
for every $t\in[t_k,1/10]$, so we again end up having \eqref{eq:refclear} for every $t\in[1/k,1/10]$.
Finally, if $t_k<1/k$, then the last monotonicity argument suffices and leads us immediately to \eqref{eq:refclear} on the whole interval $[1/k,1/10]$ again.

Therefore, it remains to establish \eqref{eq:smart5} and \eqref{eq:smart6}.

\smallskip
\emph{Proof of \eqref{eq:smart5}.}
Mathematica verifies \eqref{eq:smart5} for $10\leq k\leq 99$, so we can assume that $k\geq100$.
Using
\[ \log(1-1/k)\geq-101/100k, \quad \log k<\frac{\sqrt{k}}{2}, \quad \log_2(\pi k)+1 <\sqrt{k}, \]
and \eqref{eq:ineqforpk} we easily get
\begin{align*}
\phi_k\Big(\frac{1}{k}\Big) & = \frac{\log(2\pi)}{2} - \frac{(2k-p_k)\log k}{k} + \Big(2k-p_k+\frac{p_k}{2}-\frac{3}{2}\Big)\log\Big(1-\frac{1}{k}\Big) \\
& > \frac{\log(2\pi)}{2} - \frac{\log k}{2k} \Big(\log_2(\pi k) + 1\Big) - \frac{101}{200k}\Big(\log_2(\pi k) + 1\Big) \\
& > \frac{\log(2\pi)}{2} - \frac{601}{2000} > 0,
\end{align*}
so \eqref{eq:smart5} follows.

\smallskip
\emph{Proof of \eqref{eq:smart6}.}
Note that \eqref{eq:smart6} can be rewritten back as
\begin{equation}\label{eq:smart7}
\frac{3^{9(2k-p_k)/5+1}}{10^{2k-p_k+1}} \geq \frac{1}{\sqrt{2\pi k}}.
\end{equation}
Using \eqref{eq:ineqforpk}, recalling $k\geq10$, and observing
\[ \frac{3}{10^{41/40}} > \frac{1}{4}, \quad
\frac{9}{10}\log_2 3 - \frac{1}{2}\log_2 10 > -\frac{1}{4} \]
we can write
\begin{align*}
\frac{3^{9(2k-p_k)/5+1}}{10^{2k-p_k+1}}
& > \frac{3^{(9/10)\log_2(\pi k)+1}}{10^{(1/2)\log_2(\pi k)+41/40}}
= \frac{3}{10^{41/40}} (\pi k)^{(9/10)\log_2 3 - (1/2)\log_2 10} \\
& > \frac{1}{4} (\pi k)^{-1/4} > (2\pi k)^{-1/2},
\end{align*}
so \eqref{eq:smart7}, and thus also \eqref{eq:smart6}, is proven too.

\smallskip
\emph{Case 2: $j> k/10$.}
In this case, \eqref{eq:smart1} is the same as $\theta(1/10)\geq0$, which can be rewritten as
\begin{equation}\label{eq:smart3}
\dbinom{k}{j} 9^{(1-p_k/k)j} \Big(\frac{9}{10}\Big)^{p_k-k} \leq 1.
\end{equation}
Let us forget about the standing assumption $j>k/10$ and prove \eqref{eq:smart3} for all $1\leq j\leq k-1$.
Denote
\[ \omega_k := 9^{1-p_k/k}. \]
By direct comparison of
\[ \dbinom{k}{j-1} \omega_k^{j-1} \quad\text{and}\quad \dbinom{k}{j} \omega_k^{j} \]
we see that $\binom{k}{j}\omega_k^j$ is maximized for
\begin{equation}\label{eq:smart8}
j = j_k = \Big\lfloor \frac{\omega_k (k+1)}{1+\omega_k} \Big\rfloor.
\end{equation}
Therefore, \eqref{eq:smart3} only needs to be verified for the particular index \eqref{eq:smart8}.
Mathematica verifies \eqref{eq:smart3} for $2\leq k\leq 49$, so we can assume that $k\geq50$.

A simple auxiliary inequality
\begin{equation}\label{eq:smart9}
9^{1-t} \leq 10 (1-t)^{1-t} t^t \quad\text{for } t\in(0,1)
\end{equation}
easily follows by observing that
\[ t \mapsto (1-t)\log 9 - \log 10 - (1-t)\log(1-t) - t\log t \]
attains its maximum at $t=1/10$.
From $k\geq50$ and \eqref{eq:ineqforpk} we know that
\[ p_k > 2k - \frac{1}{2}\log_2(\pi k) - \frac{1}{200}, \]
so estimating as in \eqref{eq:Stirbinom} we get
\begin{align*}
& \dbinom{k}{j_k} 9^{(1-p_k/k)j_k} \Big(\frac{9}{10}\Big)^{p_k-k} = \dbinom{k}{j_k} 9^{(k-p_k)j_k/k} \Big(\frac{10}{9}\Big)^{k-p_k} \\
& < \frac{1}{\sqrt{2\pi}} \frac{k^{k+1/2}}{j_k^{j_k+1/2} (k-j_k)^{k-j_k+1/2}}
9^{(-k + (1/2)\log_2(\pi k) + 1/200)j_k/k} \Big(\frac{10}{9}\Big)^{-k + (1/2)\log_2(\pi k) + 1/200} \\
& < \frac{10^{1/200}}{\sqrt{2\pi k}} \frac{1}{(j_k/k)^{j_k+1/2} (1-j_k/k)^{k-j_k+1/2}}
9^{(-k + (1/2)\log_2(\pi k))j_k/k} \Big(\frac{10}{9}\Big)^{-k+(1/2)\log_2(\pi k)}.
\end{align*}
Substituting $t=j_k/k$ the last expression becomes
\begin{equation}\label{eq:smart10}
\frac{10^{1/200}}{\sqrt{2(1-t)t}} \Big(\frac{9^{1-t}}{10 (1-t)^{1-t} t^t}\Big)^{k} (\pi k)^{(1/2)(t-1)\log_2 9 + (1/2)\log_2 10 - 1/2}.
\end{equation}
Moreover
\[ -1 < 1-\frac{p_k}{k} < -\frac{9}{10}
\quad\Longrightarrow\quad \frac{1}{9} < \omega_k < \frac{1}{7}
\quad\Longrightarrow\quad \frac{41}{500} < t < \frac{51}{400}. \]
From this and \eqref{eq:smart9} we see that \eqref{eq:smart10} is at most
\[ \frac{10^{1/200}}{\sqrt{2(1-41/500)41/500}} (50\pi)^{(1/2)(51/400-1)\log_2 9 + (1/2)\log_2 10 - 1/2} < 1, \]
which finishes the proof of \eqref{eq:smart3}.

\smallskip
For the second claim in the lemma formulation we only need to use \eqref{eq:mdefoffk}, \eqref{eq:binomials}, and the binomial theorem:
\begin{equation*}
f_k(x) \leq \sum_{j=0}^{k} \dbinom{k}{j} (1-x)^{k-j} x^j = (1-x+x)^k = 1. \qedhere
\end{equation*}
\end{proof}

\begin{lemma}\label{lm:differential}
For every $x\in(0,1)$ we have
\begin{equation}\label{eq:diffeq}
a_k(x) f_k''(x) + b_k(x) f_k'(x) + p_k c_k(x) f_k(x) = 0,
\end{equation}
where $a_k,b_k,c_k\colon(0,1)\to\mathbb{R}$ are the functions defined as
\begin{align}
a_k(x) & := (1-x)^2 x^2 \big((1-x)^{p_k/k}-x^{p_k/k}\big)^2, \nonumber \\
b_k(x) & := (1-x) x \big((1-x)^{p_k/k}-x^{p_k/k}\big) \nonumber \\
& \qquad\times \Big((1-x)^{p_k/k}\big(1+2(p_k-1)x\big)+x^{p_k/k}\big(1+2(p_k-1)(1-x)\big)\Big), \nonumber \\
c_k(x) & := (1-x)^{2p_k/k} x \big(1+(p_k-1)x\big) + x^{2p_k/k} (1-x) \big(1+(p_k-1)(1-x)\big) \nonumber \\
& \qquad - (1-x)^{p_k/k} x^{p_k/k} \big(p_k-2(p_k-1)(1-x)x\big)  . \label{eq:defofck}
\end{align}
\end{lemma}

\begin{proof}
Using symbolic differentiation and algebraic simplification in Mathematica we obtain
\begin{align*}
& a_k(x) \Big(\frac{\textup{d}}{\textup{d}x}\Big)^2\big( (1-x)^{p_k(k-j)/k} x^{p_k j/k} \big)
+ b_k(x) \frac{\textup{d}}{\textup{d}x}\big( (1-x)^{p_k(k-j)/k} x^{p_k j/k} \big) \\
& \quad + p_k c_k(x) (1-x)^{p_k(k-j)/k} x^{p_k j/k} \\
& = \frac{p_k^2}{k^2} \big( (1-x)^{p_k/k} - x^{p_k/k} \big) (1-x)^{p_k(k-j)/k} x^{p_k j/k}
\big( j^2 (1-x)^{p_k/k} - (k-j)^2 x^{p_k/k} \big).
\end{align*}
It remains to multiply with $\binom{k}{j}^2$ and sum over $j=0,1,\ldots,k$ using the defining formula \eqref{eq:mdefoffk} to get
{\allowdisplaybreaks
\begin{align*}
& a_k(x) f_k''(x) + b_k(x) f_k'(x) + p_k c_k(x) f_k(x) \\
& = \frac{p_k^2}{k^2} \big( (1-x)^{p_k/k} - x^{p_k/k} \big) \\
& \quad\times \bigg( \sum_{j=0}^{k} \dbinom{k}{j}^2 j^2 (1-x)^{p_k(k-j+1)/k} x^{p_k j/k} - \sum_{j=0}^{k} \dbinom{k}{j}^2 (k-j)^2 (1-x)^{p_k(k-j)/k} x^{p_k(j+1)/k} \bigg) \\
& = p_k^2 \big( (1-x)^{p_k/k} - x^{p_k/k} \big) \\
& \quad\times \bigg( \sum_{j=1}^{k} \dbinom{k-1}{j-1}^2 (1-x)^{p_k(k-j+1)/k} x^{p_k j/k} - \sum_{j=0}^{k-1} \dbinom{k-1}{j}^2 (1-x)^{p_k(k-j)/k} x^{p_k(j+1)/k} \bigg) = 0,
\end{align*}
}
which proves \eqref{eq:diffeq}.
\end{proof}

\begin{remark}
Differential equation \eqref{eq:diffeq} can alternatively be deduced from the well-known second-order equation for the Legendre polynomials, namely
\[ (1-z^2)P_k''(z) - 2z P_k'(z) + k(k+1) P_k(z) = 0, \]
see \cite[Table~18.8.1, Row~1]{NIST}, by writing $f_k$ as
\[ f_k(x) = \big((1-x)^{p_k/k}-x^{p_k/k}\big)^k P_k\Big(\frac{(1-x)^{p_k/k}+x^{p_k/k}}{(1-x)^{p_k/k}-x^{p_k/k}}\Big). \]
In fact, this is precisely the way the author arrived at \eqref{eq:diffeq}.
\end{remark}

\begin{lemma}\label{lm:cisnegative}
If $c_k$ is defined by the formula \eqref{eq:defofck}, then for every $x\in[1/10,1/2)$ we have $c_k(x)<0$.
\end{lemma}

\begin{proof}
Substituting
\[ y=\frac{1-x}{x} \quad\Longleftrightarrow\quad x=\frac{1}{y+1} \]
the claimed inequality $c_k(x)<0$ for $x\in[1/10,1/2)$ turns into
\begin{equation}\label{eq:suff1}
\big(y^{p_k/k}-1\big) \big(p_k y^{2-p_k/k} + y^{1-p_k/k} - y - p_k\big) > 0 \quad\text{for } y\in(1,9].
\end{equation}
(Here is where Mathematica can be conveniently used too.)
Since we only care about $y>1$, inequality \eqref{eq:suff1} is further equivalent with
\begin{equation}\label{eq:suff2}
\psi_k(y) > 0 \quad\text{for } y\in(1,9],
\end{equation}
where $\psi_k\colon(0,\infty)\to\mathbb{R}$ is an auxiliary function defined as
\[ \psi_k(y) := p_k y^{2-p_k/k} + y^{1-p_k/k} - y - p_k. \]
Note that
\[ \psi_k''(y) = -\frac{p_k}{k} \Big(\frac{p_k}{k}-1\Big) y^{-1-p_k/k} \big((2k-p_k)y - 1\big), \]
which is negative, thanks to $y>1$ and \eqref{eq:ineqforpk2}.
Therefore, $\psi_k$ is concave on $[1,\infty)$ and \eqref{eq:suff2} will follow from $\psi_k(1)=0$ once we also verify that
\begin{equation*}
\psi_k(9) > 0.
\end{equation*}
However, this can be rewritten as
\begin{equation}\label{eq:suff3}
9^{2-p_k/k} > 1 + \frac{80}{9p_k+1}.
\end{equation}
For the values $2\leq k\leq 10$ Mathematica verifies \eqref{eq:suff3} by computing the ratio of the two sides reliably to $10$ digits.
On the other hand, for $k\geq11$ estimates \eqref{eq:ineqforpk} easily give $p_k>7k/4$, so, by \eqref{eq:ineqforpk},
\begin{align*}
9^{2-p_k/k}
& > \exp\Big(\frac{\log 3 \log_2(\pi k)}{k}\Big)
> 1 + \frac{\log 3 \log_2(\pi k)}{k} \\
& \geq 1 + \frac{\log 3 \log_2(11\pi)}{k}
> 1 + \frac{320}{63k} > 1 + \frac{80}{9p_k} > 1 + \frac{80}{9p_k+1}
\end{align*}
and \eqref{eq:suff3} follows.
\end{proof}

Now we are in position to give a short proof of the main result.

\begin{proof}[Proof of Theorem~\ref{thm:main}]
Since $f_k(1-x)=f_k(x)$, we only need to show the desired inequality \eqref{eq:mainineq2} for $x\in[0,1/2]$.
Suppose that the maximum of $f_k$ on $[0,1/2]$ is strictly greater than $1$ and that this maximum is attained at some point $x_{\max}$.
From Lemma~\ref{lm:easyest} and $f_k(1/2)=1$ we conclude $1/10<x_{\max}<1/2$.
Necessary conditions for the local maximum give $f_k'(x_{\max})=0$ and $f_k''(x_{\max})\leq0$.
We clearly have $a_k(x_{\max})\geq0$, while Lemma~\ref{lm:cisnegative} gives $c_k(x_{\max})<0$.
Finally, differential equation \eqref{eq:diffeq} from Lemma~\ref{lm:differential} gives a contradiction:
\[ 0 = \underbrace{a_k(x_{\max})}_{\geq0} \underbrace{f_k''(x_{\max})}_{\leq0} + b_k(x_{\max}) \underbrace{f_k'(x_{\max})}_{=0} + p_k \underbrace{c_k(x_{\max})}_{<0} \underbrace{f_k(x_{\max})}_{>1} < 0. \qedhere \]
\end{proof}

\begin{remark}
Since we were using the trick of bounding a function via a differential equation, one can still ask for a ``more quantitative'' proof of estimate \eqref{eq:mainineq2}.
Using \eqref{eq:Stirbinom} and estimating the error in Simpson's formula, one can show that
$f_k(x) \leq g_k(x) + O(1/k)$
for $1/10\leq x\leq1/2$ as $k\to\infty$, where
\[ g_k(x) := \frac{1}{2\pi} \int_{1/k}^{1-1/k} \varphi_{k,x}(t) \,\textup{d}t \]
and $\varphi_{k,x}$ is an auxiliary function defined by the formula
\[ \varphi_{k,x}(t) := \frac{(1-x)^{(1-t)p_k} x^{t p_k}}{(1-t)^{2(1-t)k+1} t^{2tk+1}}. \]
Moreover, the function $\varphi_{k,x}$ is log-concave on $[1/k,1-1/k]$, attains its unique maximum on that interval at some point $t_{\max}$ between $x$ and $1/2$, and it can be estimated pointwise in terms of its maximum as
\[ \varphi_{k,x}(t) \leq \varphi_{k,x}(t_{\max}) e^{-4(k-1)(t-t_{\max})^2} \]
for $1/k\leq t\leq1-1/k$.
All this can be turned into an alternative proof of \eqref{eq:mainineq2} on a major part of the interval $[0,1/2]$, such as $1/10\leq x\leq 1/2-1/\sqrt{k}$, but only for sufficiently large integers $k$.
\end{remark}


\section{Proof of Corollary~\ref{cor:probab}}
\label{sec:corollary}
Take an arbitrary parameter $q\in[0,1]$ and let $Y_1,Z_1,Y_2,Z_2,\ldots$ be independent random variables with distributions
\[ Y_i \sim \left(\begin{matrix} 0 & 1 \\ 1-q & q \end{matrix}\right), \quad Z_i \sim \left(\begin{matrix} 0 & 1 \\ 1/2 & 1/2 \end{matrix}\right) \]
for every index $i=1,2,\ldots$.
Observe that
\[ Y_i - Z_i \sim \left(\begin{matrix} -1 & 0 & 1 \\ (1-q)/2 & 1/2 & q/2 \end{matrix}\right), \]
so by choosing $q=2\mathbb{P}(X_1=1)$ we achieve that the sequence $(Y_i-Z_i)_{i=1}^{\infty}$ has the same (joint) distribution as $(X_i)_{i=1}^{\infty}$.
Inequality \eqref{eq:probineq} now becomes
\begin{align*}
& \bigg(\sum_{j=0}^{k} \mathbb{P}(Y_1+\cdots+Y_k=j,\, Z_1+\cdots+Z_k=j)\bigg)^{1/p_k} \\
& \leq \mathbb{P}(Y_1=\cdots=Y_k=0,\, Z_1=\cdots=Z_k=1)^{1/p_k} \\
& \quad + \mathbb{P}(Y_1=\cdots=Y_k=1,\, Z_1=\cdots=Z_k=0)^{1/p_k}.
\end{align*}
This simplifies further as
\begin{equation}\label{eq:altprob}
\bigg(\sum_{j=0}^{k} \dbinom{k}{j} (1-q)^{k-j} q^j \dbinom{k}{j}\Big(\frac{1}{2}\Big)^k \bigg)^{1/p_k} \leq \Big(\frac{1-q}{2}\Big)^{k/p_k} + \Big(\frac{q}{2}\Big)^{k/p_k}
\end{equation}
for every positive integer $k$ and every $q\in[0,1]$, which becomes precisely \eqref{eq:mainineq} with
\begin{equation}\label{eq:altparam}
a=\Big(\frac{1-q}{2}\Big)^{k/p_k}, \quad b=\Big(\frac{q}{2}\Big)^{k/p_k}.
\end{equation}

Conversely, by the homogeneity of \eqref{eq:mainineq} one is allowed to add an additional constraint $a^{p_k/k}+b^{p_k/k}=1/2$, which allows us to parameterize the pair $(a,b)$ as in \eqref{eq:altparam} for some $q\in[0,1]$ and then \eqref{eq:mainineq} turns precisely into \eqref{eq:altprob}.
The claim about the optimality of $p_k$ then also follows from the initial comments on the sharpness of \eqref{eq:mainineq}.


\section{Comments on number means}
\label{sec:means}
For a positive integer $k$ the quantity
\[ \mathfrak{W}_2^{[k,k]}(x,y) := \bigg( \frac{\sum_{j=0}^{k}\binom{k}{j}^2 x^{k-j} y^{j}}{\binom{2k}{k}} \bigg)^{1/k} \]
is \emph{the Whiteley mean} \cite[Subsection~V.5.2]{Bul03} with parameters $k,k$ of the numbers $x,y\in[0,\infty)$.
It can also be understood as \emph{the $k$-th elementary symmetric polynomial mean} \cite[Section~V.1]{Bul03} of the numbers
\[ \underbrace{x,\ldots,x}_{k},\underbrace{y,\ldots,y}_{k}. \]
On the other hand,
\[ \mathfrak{M}_2^{[r]}(x,y) :=
\begin{cases}
\bigg( {\displaystyle\frac{x^r + y^r}{2}} \bigg)^{1/r} & \text{for } r\in\mathbb{R},\ r\neq0, \\
\sqrt{x y} & \text{for } r=0
\end{cases} \]
is the well-known \emph{power mean} \cite[Section~III.1]{Bul03} with exponent $r$.
Substituting $x=a^{p_k/k}$, $y=b^{p_k/k}$ the main inequality of this paper \eqref{eq:mainineq} can be reformulated equivalently as
\begin{equation}\label{eq:meansineq1}
\mathfrak{W}_2^{[k,k]}(x,y) \leq \mathfrak{M}_2^{[r_k]}(x,y) \quad\text{for } x,y\in[0,\infty),
\end{equation}
where
\[ r_k = \frac{k}{\log_2\binom{2k}{k}} \in\Big(\frac{1}{2},1\Big]. \]
Choosing $x=1$, $y=0$ one easily observes that this exponent $r_k$ is the smallest one such that \eqref{eq:meansineq1} can hold.
A very special case of a result by Bochi, Iommi, and Ponce \cite[Theorem~3.4]{BIP21} showed a weaker inequality,
\[ \mathfrak{W}_2^{[k,k]}(x,y) \leq 2^{2-1/r_k} \mathfrak{M}_2^{[1/2]}(x,y), \]
which would not be sufficient for our intended application to Corollary~\ref{cor:additive}, but it becomes the same as \eqref{eq:meansineq1} in the limit as $k\to\infty$.
Nice observations from this paragraph have all been communicated to the author by Jairo Bochi.

In the other direction, Bochi, Iommi, and Ponce \cite[Theorem~3.4]{BIP21} also proved
\[ \mathfrak{W}_2^{[k,k]}(x,y) \geq \mathfrak{M}_2^{[1/2]}(x,y). \]
We are in position to give a sharpening of the last estimate:
\begin{equation}\label{eq:meansineq2}
\mathfrak{W}_2^{[k,k]}(x,y) \geq \mathfrak{M}_2^{[k/(2k-1)]}(x,y) \quad\text{for } x,y\in[0,\infty)
\end{equation}
and every positive integer $k$.
In fact, the exponent $r=k/(2k-1)$ is the largest one such that inequality $\mathfrak{W}_2^{[k,k]}(x,y) \geq \mathfrak{M}_2^{[r]}(x,y)$ can hold; this is easily seen by observing the asymptotic expansions:
\begin{align*}
\mathfrak{W}_2^{[k,k]}(1+\varepsilon,1-\varepsilon) & = 1 - \frac{k-1}{2(2k-1)}\varepsilon^2 + O(\varepsilon^3), \\
\mathfrak{M}_2^{[r]}(1+\varepsilon,1-\varepsilon) & = 1 - \frac{1-r}{2}\varepsilon^2 + O(\varepsilon^3)
\end{align*}
as $\varepsilon\to0^+$.
Therefore, \eqref{eq:meansineq1} and \eqref{eq:meansineq2} together give optimal comparisons of the Whiteley mean $\mathfrak{W}_2^{[k,k]}(x,y)$ with power means.

The proof of inequality \eqref{eq:meansineq2} is much simpler than that of \eqref{eq:mainineq}. We only need to reuse a few ideas from Section~\ref{sec:theorem}.
Substituting $a=x^{k/(2k-1)}$, $b=y^{k/(2k-1)}$ we transform \eqref{eq:meansineq2} into
\[ \sum_{j=0}^{k} \dbinom{k}{j}^2 a^{(2k-1)(k-j)/k} b^{(2k-1)j/k} \geq \dbinom{2k}{k} 2^{-2k+1} (a+b)^{2k-1} \quad\text{for } a,b\in[0,\infty). \]
By homogeneity we can normalize $a+b=1$.
Thus, we only need to prove
\[ h_k(x) \geq h_k\Big(\frac{1}{2}\Big) \quad \text{for every } x\in[0,1] \]
and a function $h_k\colon[0,1]\to[0,\infty)$ defined by
\[ h_k(x) := \sum_{j=0}^{k} \dbinom{k}{j}^2 (1-x)^{(2k-1)(k-j)/k} x^{(2k-1)j/k}. \]
Recall that the proof of Lemma~\ref{lm:differential} did not use the particular value of $p_k$, so it can be applied with $p_k$ replaced with $2k-1$. This observation yields the differential equation
\begin{equation}\label{eq:diffeqhk}
\widetilde{a}_k(x) h_k''(x) + \widetilde{b}_k(x) h_k'(x) + (2k-1) \widetilde{c}_k(x) h_k(x) = 0,
\end{equation}
where
\begin{align*}
\widetilde{a}_k(x) & := (1-x)^2 x^2 \big((1-x)^{2-1/k}-x^{2-1/k}\big)^2, \\
\widetilde{b}_k(x) & := (1-x) x \big((1-x)^{2-1/k}-x^{2-1/k}\big) \\
& \qquad\times \Big((1-x)^{2-1/k}\big(1+4(k-1)x\big)+x^{2-1/k}\big(1+4(k-1)(1-x)\big)\Big), \\
\widetilde{c}_k(x) & := (1-x)^{4-2/k} x \big(1+2(k-1)x\big) + x^{4-2/k} (1-x) \big(1+2(k-1)(1-x)\big) \\
& \qquad - (1-x)^{2-1/k} x^{2-1/k} \big(2k-1-4(k-1)(1-x)x\big) .
\end{align*}
Note that for $k\geq2$ and $x\in(0,1/2)$ we have $\widetilde{a}_k(x)>0$ and $\widetilde{c}_k(x)>0$. Indeed, by substituting $z=((1-x)/x)^{1/k}>1$, simplifying, and factoring polynomials, positivity of $\widetilde{c}_k$ reduces to
\[ \frac{z^{2k-1} (z-1) (z^{2k-1}-1)}{(z^k+1)^{6-2/k}} \Big(\sum_{i=1}^{k-1}(z^i + z^{-i} - 2)\Big) > 0, \]
which clearly holds.
It is easy to see $h_k(1-x)=h_k(x)$ and $h_k(0)\geq h_k(1/2)$.
Now take $k\geq2$ and suppose that $h_k$ attains its minimum at some point $x_{\min}\in(0,1/2)$.
Differential equation \eqref{eq:diffeqhk} gives
\[ 0 = \underbrace{\widetilde{a}_k(x_{\min})}_{>0} \underbrace{h_k''(x_{\min})}_{\geq0} + \widetilde{b}_k(x_{\min}) \underbrace{h_k'(x_{\min})}_{=0} + (2k-1) \underbrace{\widetilde{c}_k(x_{\min})}_{>0} \underbrace{h_k(x_{\min})}_{>0} > 0, \]
which is a contradiction.


\section*{Acknowledgments}
This work was supported in part by the \emph{Croatian Science Foundation} project IP-2018-01-7491 (DEPOMOS).
The author is grateful to Aleksandar Bulj for turning his attention to the problem and to Rudi Mrazovi\'{c} for bringing up lazy simple random walks in a discussion.
The author also thanks Jairo Bochi for excellent remarks on number means and an elegant reformulation of the main inequality.
Finally, the author is grateful to the anonymous referee for discovering a minor mistake in the proof of Lemma~\ref{lm:easyest}.


\bibliography{sums_and_energies}{}
\bibliographystyle{plain}

\end{document}